\documentclass[a4paper]{amsart}

\usepackage{dsfont}
\usepackage{amsmath, amsthm, amssymb}
\usepackage{xypic}
\usepackage{graphicx}
\usepackage{color}
\usepackage[usenames,dvipsnames]{xcolor}
\usepackage[latin1]{inputenc}
\usepackage{indentfirst}
\usepackage{epstopdf}
\usepackage{subfigure}
\usepackage{emptypage}
\usepackage{hyperref}
\hypersetup{
    colorlinks=true, %set true if you want colored links
    linktoc=all,     %set to all if you want both sections and subsections linked
    linkcolor=blue,  %choose some color if you want links to stand out
    citecolor=JungleGreen,
}

\newtheorem{theorem}{Theorem}%[section]

\newtheorem{proposition}[theorem]{Proposition}
\newtheorem{lemma}[theorem]{Lemma}
\newtheorem{definition}[theorem]{Definition}

\newtheorem{theorem*}{Theorem}
\newtheorem{question*}[theorem*]{Question}
\newtheorem{conjecture*}[theorem*]{Conjecture}
\newtheorem{corollary*}[theorem*]{Corollary}
\newtheorem{theorem*e}{Teorema}
\newtheorem{question*e}[theorem*e]{Pregunta}
\newtheorem{conjecture*e}[theorem*e]{Conjetura}
\newtheorem{corollary*e}[theorem*e]{Corolario}

\newcommand{\interior}{\mathrm{int}}

\newcommand{\id}{\mathrm{id}}
\newcommand{\dist}{\mathrm{dist}}
\newcommand{\diam}{\mathrm{diam}}

\newcommand{\R}{\mathds{R}}
\newcommand{\Z}{\mathds{Z}}
\newcommand{\A}{\mathds{A}}
\newcommand{\T}{\mathds{T}}

\title{An elementary proof of a Theorem by Matsumoto}

\author{Luis Hernández--Corbato}
\address{IMPA, Estrada dona Castorina 110, Rio de Janeiro, Brazil.}
\email{luishcorbato@mat.ucm.es}


\begin{document}

\maketitle

\begin{abstract}
Matsumoto proved in \cite{matsumoto} that the prime end rotation numbers associated to an invariant annular continuum
are contained in its rotation set. An alternative proof of this fact using only simple planar topology is presented.
\end{abstract}

\section{Introduction}

The rotation number was introduced by Poincaré to study the dynamics of circle homeomorphisms
$f\colon S^1 \to S^1$. Given a lift $\tilde{f} \colon \R \to \R$ of $f$, the rotation number of $\tilde{f}$ is defined as
$\rho(\tilde{f}) = \lim_{n \to \infty} (\tilde{f}^n(x) - x)/n \in \R$, for any $x \in \R$.
The limit is independent of $x$ and only depends on the lift $\tilde{f}$ up to an integer constant.
The rotation number $\rho(f) = \rho(\tilde{f}) \mod \Z \in \R/\Z$ measures the speed at which points
rotate under the iteration of $f$ and essentially classifies the dynamics.

The definition of rotation number does not extend smoothly to homeomorphisms of the annulus
$f \colon S^1 \times [-1,1] = \A \rightarrow \A$. Consider the universal cover of $\A$ identified to $\R \times [-1, 1]$
and let $\tilde{f} \colon \R \times [-1, 1] \to \R \times [-1,1]$ be a lift of $f$.
Denote $(x)_1$ the first coordinate of a point $x \in \R \times [-1,1]$.
Then $\lim_{n \to \infty}((\tilde{f}^n(x))_1 - (x)_1)/n$ now depends on $x$ and, even worse, may not exist.
Instead of looking at orbits it is useful to consider $f$--invariant probabilities $\mu$ in $\A$ and define
$$\rho(\tilde{f}, \mu) = \int_{\A} (\tilde{f}(s(y)))_1 - (s(y))_1 d\mu (y),$$
where $s\colon S^1 \times [-1,1] \to \R \times [-1,1]$ is a section of the universal cover
$\pi: \R \times [-1, 1] \to S^1 \times [-1, 1]$, i.e. $s \pi = \id$.
%where $\phi(z) = (\tilde{f}(x))_1 - (x)_1$ for any $x \in
More generally, denote
$$\rho_{mes}(\tilde{f}, X) =
  \{\rho(\tilde{f}, \mu): \mu\text{ is an }f\text{--invariant Borel probability and } supp(\mu) \subset X\},$$
%$$\rho_{mes}(\tilde{f}, X) = \{\rho(\tilde{f}, \mu): f_*\mu = \mu, supp(\mu) \subset X\},$$
for any $f$--invariant set $X \subset \A$.
Since the space of invariant Borel probabilities endowed with the weak topology is compact and convex
it follows that $\rho_{mes}(\tilde{f}, X)$ is a compact interval.

A continuum $X \subset \interior(\A)$ is \emph{essential} if the two boundary components $S^1 \times \{-1\}$
and $S^1 \times \{1\}$ of $\A$ belong to different connected components,
denoted, respectively, $U_-$ and $U_+$, of $\A \setminus X$.
It is called an (essential) \emph{annular} continuum if, additionally, $\A \setminus X = U_- \cup U_+$.
The previous notions of rotation can be applied to study the dynamics of invariant annular continua.
In contrast with the one--dimensional case, the coexistence of different rotation numbers is typical.
An example is the Birkhoff attractor $\Lambda$ \cite{birkhoff},
which is the global attractor of a dissipative diffeomorphism of the open annulus.
Even though $\Lambda$ has empty interior, it contains infinitely many periodic orbits with different rotation numbers.

There is yet another way of measuring the rotation of an invariant annular continuum. After identifying $S^1 \times \{1\}$
to a point, $U_+$ is transformed into an invariant open topological disk.
Carathéodory's prime end theory (see \cite{mather})
permits to compactify this new domain with a boundary circle, the set of prime ends of $U_+$,
producing a closed topological disk $\hat{U}_+$.
The construction being topological allows the homeomorphism $f$ to be extended to a homeomorphism
$\hat{f}\colon \hat{U}_+ \to \hat{U}_+$.
Furthermore, a lift $\tilde{f}$ of $f$ uniquely determines a lift $\hat{F} \colon \R \to \R$
of the restriction of $\hat{f}$ to the circle of prime ends of $U_+$, boundary of $\hat{U}_+$, and viceversa.
The \emph{upper prime end rotation number} of the lift $\tilde{f}$ in $X$ is defined as the rotation number
of $\hat{F}$ and denoted $\rho_+(\tilde{f}, X)$. The \emph{lower prime end rotation number} $\rho_-(\tilde{f}, X)$
is defined analogously. One can think of these rotation numbers as measures of the rotation of the boundary of $X$
as seen from the exterior.

An alternative intuitive approach to the prime end rotation numbers in terms of accessible points is discussed in \cite{barge}.
A point $p$ is called \emph{accessible} from a domain $U$, $p \notin U$, provided
there is an arc $\gamma\colon [0,1] \to U \cup \{p\}$ such that $\gamma([0,1)) \subset U$ and $\gamma(1) = p$.
Denote $\tilde{U}_+, \tilde{X}$ the lifts of $U_+, X$ to the universal cover $\R \times [-1, 1]$ of $\A$.
Let $x \neq x' \in \tilde{X}$ be accessible from $U_+$ and $\gamma, \gamma' \colon [0,1] \to \R \times [-1,1]$
be two disjoint arcs such that $\gamma(0), \gamma'(0) \in \R \times \{1\}$, $\gamma([0,1)), \gamma'([0,1)) \subset \tilde{U}_+$
and $\gamma(1) = x$ and $\gamma'(1) = x'$.
Denote $\gamma(0) = (r, 0), \gamma'(0) = (r', 0)$ and define $x \prec x'$ if and only if $r < r'$.
Then, $\prec$ defines a linear order in the set of points of $X$ accessible from $U_+$.
For any $x, y$ in that set and $n \in \Z$ there is a unique $k = k(x, y)$ such that $T^k(y) \preceq x \prec T^{k+1}(y)$,
where $T$ denotes the deck transformation of the universal cover.
One can prove that $\lim_{n \to \infty} (k(\tilde{f}^n(x), y))/n$
is independent of $x$ and $y$ and is equal to $\rho_+(\tilde{f}, X)$.

%The equality follows from a standard fact
%of prime end theory, namely that arcs ending at an accessible point are in one--to--one correspondence to
%arcs accessing points in the prime end circle

The goal of this article is to give an elementary proof of the following theorem
due to Matsumoto \cite{matsumoto}.
\begin{theorem}[Matsumoto]\label{thm:}
Let $f\colon \A \to \A$ be a homeomorphism isotopic to the identity and $X \subset \interior(\A)$
an invariant annular continuum. For any lift $\tilde{f}$ of $f$
$$\rho_+(\tilde{f}, X), \rho_-(\tilde{f}, X) \in \rho_{mes}(\tilde{f}, X).$$
\end{theorem}
Recall that a classical result due to Epstein shows that $f: \A \to \A$ is isotopic to the identity
if and only if preserves orientation and each of the boundary circles.

Matsumoto's proof of Theorem \ref{thm:}
uses Le Calvez's deep theorem on the existence of a foliation by Brouwer lines for any orientation
preserving homeomorphism of $\R^2$ in its equivariant form \cite{lecalvez} for the torus $\T^2$.
The proof then goes on concluding the result in each of several cases, depending on the topological type of the aforementioned
foliation. In this paper an alternative proof of Theorem \ref{thm:} is presented.
The arguments involve only basic facts from planar topology and prime end theory
making our approach elementary in nature.

Theorem \ref{thm:} allows to estimate the size of the rotation set of $X$, $\rho_{mes}(\tilde{f}, X)$, without
precise information of the dynamics within $X$. It can be subsequently applied to conclude the existence
of periodic orbits in $X$ of any rotation number
$p/q \in [\rho_{\pm}(\tilde{f}, X),  \rho_{\mp}(\tilde{f}, X)] \subset \rho_{mes}(\tilde{f}, X)$
provided some extra hypothesis is satisfied:
either $f_{X}$ is chain--recurrent (polishing an argument due to Franks \cite{franks}, see \cite{koro, matsumoto})
or $f$ is area--preserving \cite{frankslecalvez}
or $X$ is a cofrontier \cite{barge} or, more generally, a circloid \cite{koro}.

In order to ease the notation, for any integer $k$ the action $T^k(S)$ of the deck transformation $T$
on a set $S \subset \R \times [-1, 1]$ will be denoted $S + k$. Additionally, the projection $p(S)$ of $S$ under
the first coordinate map $p: \R \times [-1, 1] \to \R$ will be denoted $(S)_1$.

\subsection*{Acknowledgements}
The author gratefully acknowledges Prof. Shigenori Matsumoto for pointing
out an error in a draft version of this work and warmly thanks Prof. Andrés Koropecki for useful conversations
which brought Matsumoto's work to attention. The author has been supported by CNPq (Conselho Nacional de Desenvolvimento Científico e Tecnológico - Brasil) and partially by MICINN grant MTM2012-30719.
As an aside, special compliments are sent to the Spanish administration for its diligence.

\section{Proof of Theorem \ref{thm:}}

Next lemma follows directly from the definitions.

\begin{lemma}\label{lem:rotationnumbers}
For any integer $k$,
$$\rho_{mes}(T^k \tilde{f}^n, X) = n \rho_{mes}(\tilde{f}, X) + k, \enskip \enskip
\rho_{\pm}(T^k \tilde{f}^n, X) = n \rho_{\pm}(\tilde{f}, X) + k.$$
\end{lemma}

The proof of Theorem \ref{thm:} presented here only deals with the upper prime rotation number and shows that
$\rho_+(\tilde{f}) \ge \inf \rho(\tilde{f}, X)$, the other cases being completely analogous.

Argue by contradiction: suppose there are integers $p, q$ such that $\rho_+(\tilde{f}) < p/q < \inf \rho(\tilde{f}, X)$.
As a consequence of Lemma \ref{lem:rotationnumbers}, $\rho_+(T^{-p}\tilde{f}^q) < 0 < \inf \rho(T^{-p}\tilde{f}^q, X)$.
Thus, after renaming, it is possible to assume
\begin{equation}\label{eq:}
\rho_+(\tilde{f}) < 0 < \inf \rho(\tilde{f}, X).
\end{equation}

Some notation to describe the shape of $\tilde{U}_+$ is now introduced.
Let $\eta = \max\{y \in \R: (0, y) \in \tilde{X}\}$
and $\beta\colon [0,1] \to \R \times [-1, 1]$ be a vertical arc with endpoints $\beta(0) = (0,1) \in \R \times \{1\}$
and $\beta(1) = (0, \eta) \in \tilde{X}$. Denote $x_0 = \beta(1)$ and use $\beta$ also to
denote the image of the arc $\beta$. This abuse of notation is present throughout the text.
The arcs $\beta' = \beta \setminus \{x_0\}$ and $\beta' + 1$ bound a region in $\tilde{U}_+$
which contains the segment $(0,1) \times \{1\}$.
The closure of this region, as a subset of $\tilde{U}_+$, will be denoted $V$ and thought of as a fundamental region.
Clearly,
$$\tilde{U}_+ = \bigcup_{k \in \Z} (V + k)$$
and $(V - 1) \cap V = \beta'$.
Define $$(V + k)^+ = \bigcup_{j \ge k} (V+j), \enskip (V + k)^- = \bigcup_{j \le k} (V+j).$$

Note that $x_0 \in \tilde{X}$ is accessible both from $\tilde{U}_+$ and from $V$.
The following lemma is based on the interpretation of the prime end rotation number in terms of
accessible points and their induced order $\prec$ as was discussed in the introduction.

\begin{lemma}\label{lem:primeends}
\mbox{}
\begin{enumerate}
\item For every point $x$ accessible from $\tilde{U}_+$ there exists $k \in \Z$ such that $x$ is accessible from $V+k$.
\item If $x$ is accessible from $V+k$ then $\tilde{f}^{n}(x)$ is accessible from $(V+k)^+$, if $n \le 0$, or
from $(V+k)^-$, if $n \ge 0$.
\item Suppose $x \in \tilde{X}$ is accessible from $V+k_1$ and $\tilde{f}^{-1}(x)$ is accessible from $V + k_2$.
Then, for any point $z \in \tilde{X}$ accessible from $\tilde{U}_+$
there exists an integer $n$ so that $\tilde{f}^n(z)$ is accessible from $V+j$, for some $k_1 \le j \le k_2$.
\end{enumerate}
\end{lemma}

A \emph{crosscut} of $\tilde{U}_+$ is an arc $c$
whose endpoints lie in $\tilde{X}$ and whose interior is contained in $\tilde{U}_+$.
By definition, the endpoints of $c$ are accessible from $\tilde{U}_+$.
Recall a standard fact from prime end theory: $c$ separates $\tilde{U}_+$ in exactly two connected components.

\begin{definition}
An arc $\gamma: [0,1] \to V$ is said to be a \emph{hair} of $V$ if $\gamma(0) \in \R \times \{1\}$.
More generally, an arc $\gamma: [0,1] \to \tilde{U}_+$ is a \emph{hair} if $\gamma - k$ is a hair of $V$ for some $k \in \Z$.
In that case $\gamma$ is called a hair of $V + k$.
\end{definition}

\begin{figure}[htb]
\begin{center}
\includegraphics[scale = 1]{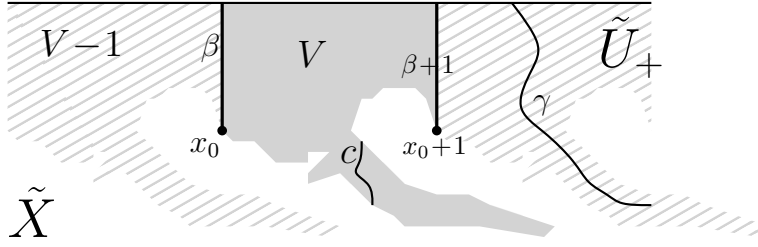}
\end{center}
\caption{$c$ is a crosscut of $\tilde{U}_+$ and $\gamma$ is a hair of $V + 1$.}
\end{figure}

\begin{lemma}\label{lem:aladerecha}
%For every $x \in \tilde{X}$, the sequence $\{(\tilde{f}^n(x))_1\}_{n \ge 1}$ tends to $+\infty$.
%As a consequence, there exist $m \ge 1$ and $\epsilon > 0$ so that
%$(\tilde{f}^m(x))_1 \ge (x)_1 + \epsilon$ for every $x \in \tilde{X}$.
There exist $m \ge 1$ satisfying $(\tilde{f}^m(x))_1 \ge (x)_1 + 1$ for every $x \in \tilde{X}$.
Furthermore, there exists $M > 1$ such that if $x$ is a point in $V$ for which every hair $\gamma$ of $V$
ending at $x$ satisfies $\diam((\gamma)_1) > M$ then $(\tilde{f}^m(x))_1 \ge (x)_1 + 1 / 2$.
\end{lemma}
\begin{proof}
For the first part, suppose on the contrary that there are integers $\{n_i\}_{i \ge 1} \to +\infty$
and points $\{x_i\}_{i \ge 1}$ in $\tilde{X}$ such that $(\tilde{f}^{n_i}(x_i))_1 < (x_i)_1 + 1$.
The probability measures defined on $X$ by
$$\mu_{i} = \frac{1}{n_i} \sum_{j = 0}^{n_i-1} \delta_{f^j(\pi(x_i))}$$
satisfy $\rho(\tilde{f}, \mu_{i}) = (\tilde{f}^{n_i}(x_i))_1/n_i - (x_i)_1/n_i < 1/n_i$.
The space of Borel probability measures on $X$ endowed with the weak topology being compact and metric,
there is a subsequence
$\{\mu_{i_j}\}_{j \ge 1}$, $i_j \to +\infty$, of $\{\mu_i\}$ whose limit is a Borel probability measure $\mu$.
By continuity of the pushforward operator
\[
f_*(\mu) - \mu = \lim_j f_*(\mu_{i_j}) - \mu_{i_j} =
\lim_j \frac{1}{n_{i_j}} (\delta_{f^{n_{i_j}}(\pi(x_i))} - \delta_{\pi(x_i)}) = 0,
\]
$\mu$ is $f$--invariant, and by the weak convergence $\mu_{i_j} \to \mu$
\[
\rho(\tilde{f}, \mu) = \lim_j \rho(\tilde{f}, \mu_{i_j}) \le \lim_j 1/n_{i_j} = 0,
\]
which contradicts $(\ref{eq:})$.

For the second, let $\tilde{Y} = \{z \in \R \times [-1, 1]: (\tilde{f}^m(z))_1 \ge (z)_1 + 1/2\}$
be a neighborhood of $\tilde{X}$.
Clearly, the projection $Y = \pi(\tilde{Y})$ of $\tilde{Y}$ onto $\A$ is a compact neighborhood of $X$ and
$\dist(X, \partial Y) = \delta > 0$.
If the statement does not hold one can find points $z_n \in V \setminus \tilde{Y}$ such that any hair $\gamma$ of $V$
ending at $z_n$ satisfies $\diam((\gamma)_1) > n$, for every $n \ge 1$. It is possible to choose an infinite subsequence
$\{z_{n_j}\}_j$ of $\{z_n\}_n$ so that the balls centered at $z_{n_j}$ of radius $\delta$ are pairwise disjoint
and contained in $V$. This is impossible because $V$ has finite area.
\end{proof}

For simplicity, for the rest of the proof replace $f$ by $f^m$, where $m$ is as in the previous lemma.
Then, $(\tilde{f}(x))_1 \ge (x)_1 + 1$ for any $x \in \tilde{X}$ and Inequality (\ref{eq:}) still holds.

The following object gives a way to roughly describe the shape of $V$. Construct hairs $\gamma_n$, $n \ge 1$, in $V$
such that $l_{n+1} < l_n$, $r_{n+1} > r_n$, where $l_n = \min (\gamma_n)_1$ and $r_n = \max (\gamma_n)_1$.
%Note that this construction may yield a finite or infinite sequence.

There are three mutually exclusive cases depending on $V$:
\begin{itemize}
\item[(i)] It is not possible to have $\lim_n l_n = - \infty$.
\item[(ii)] There is an infinite sequence of hairs $\{\gamma_n\}_n$
such that $\lim_n l_n = - \infty$ and $\{r_n\}_n$ is bounded.
\item[(iii)] For any infinite sequence of hairs $\{\gamma_n\}_n$
such that $\lim_n l_n = - \infty$, always $\lim_n r_n = + \infty$.
\end{itemize}

They correspond to:
(i) $(V)_1$ is bounded from below, (ii) $(V)_1$ unbounded from below but bounded from above and (iii)
$(V)_1$ is unbounded both from below and above.

The proof of Theorem \ref{thm:} deals separately with these three cases. Lemmas \ref{lem:primeends} and \ref{lem:aladerecha}
are extensively used to derive a contradiction with Inequality \ref{eq:} in each of them.

\bigskip

\textbf{Case (i): $(V)_1$ is bounded from below.}

In this case there exists $L \in \R$ such that every point $x \in \tilde{X}$ accessible from $V$ satisfies $L < (x)_1$.
Consider $k$ so that $\tilde{f}(x_0)$ is accessible from $V + k$. Lemma \ref{lem:primeends} ensures $k \le 0$.
It follows that for any point $z \in \tilde{X}$ accessible from $V^-$ there exists $n \ge 0$ such that
$\tilde{f}^{-n}(z)$ is accessible from $V + j$, for some $k \le j \le 0$.
The lower bound on $(V)_1$ implies $L + k < (\tilde{f}^{-n}(z))_1$
so, by Lemma \ref{lem:aladerecha}, $L + k < (z)_1$, which is absurd. \qed

\bigskip

\textbf{Case (ii): $(V)_1$ is unbounded from below but bounded from above.}

The union of $\gamma_n$, $n \ge 1$, is contained in $V$ and divides $\tilde{U}_+$ in many connected components. Denote $B$
the component which contains $(V - 1)^-$. Clearly, $B$ is unbounded from the left
but satisfies $\sup(B)_1 \le M$, where $M$ is an upper bound for $\{r_n\}_n$.

Let $x \in \tilde{X}$ be accessible from $V - 1 \subset B$.
For every $n \ge 1$, by Lemma \ref{lem:primeends}
the point $\tilde{f}^n(x)$ is accessible from $(V - 1)^-$ and thus from $B$ as well.
However, for large $n \ge 0$, Lemma \ref{lem:aladerecha} implies $(\tilde{f}^n(x))_1 > M$
and, in particular, $\tilde{f}^n(x)$ cannot belong to the adherence of $B$. \qed

\bigskip

\textbf{Case (iii): $(V)_1$ is unbounded both from below and above.}

This case is more involved and some preliminary results are needed.
Firstly, the shape of the region $V$ is shown to be snake--like. This idea is made precise in the following proposition.
\begin{proposition}\label{prop:}
There exist sequences $\{L_n\}_n, \{R_n\}_n$ of real numbers such that
\begin{enumerate}
\item $R_1 > 1, L_1 < 0$.
\item $\{L_n\}_n$ is decreasing and tends to $- \infty$.
\item $\{R_n\}_n$ is increasing and tends to $+ \infty$.
\end{enumerate}
such that
\begin{itemize}
\item[(i)] If $\gamma$ is a hair of $V$ and $(\gamma(1))_1 < L_n$ then $R_n \in (\gamma)_1$.
\item[(ii)] If $\gamma$ is a hair of $V$ and $(\gamma(1))_1 > R_n$ then $L_{n-1} \in (\gamma)_1$.
\end{itemize}
\end{proposition}
\begin{proof}
Let $\alpha, \gamma$ be hairs of $V$. Notice the following simple fact:
$\gamma$ does not intersect any translated $\alpha + k$ unless $k = 0$. As a consequence:
\begin{itemize}
\item[($\star$)] If $(\gamma(1))_1 > \max(\alpha)_1 + 1$ then $\min(\gamma)_1 < \min(\alpha)_1 + 1$.
\item[($\star\star$)] If $(\gamma(1))_1 < \min(\alpha)_1 - 1$ then $\max(\gamma)_1 > \max(\alpha)_1 - 1$.
\end{itemize}
\begin{figure}[htb]
\begin{center}
\includegraphics[scale = 0.7]{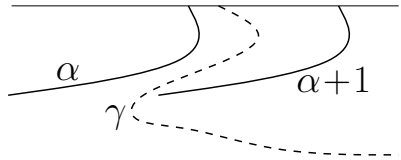}
\end{center}
\caption{Figure for statement ($\star$)}
\end{figure}
The proof goes on following a mechanical routine using the sequence of hairs $\{\gamma_n\}_n$
and the associated scalar sequences $\{l_n\}_n, \{r_n\}_n$.
First, choose $m_1$ so that $r_{m_1} > 2$ and $l_{m_1} < 1$.
Take $R_1 = r_{m_1} - 1$ and $L_1 = l_{m_1} - 1$ and note that by ($\star\star$) with $\alpha = \gamma_{m_1}$
the statement (i) holds for $n = 1$.
Then, take $m'_2$ so that $l_{m'_2} < L_1 - 1$ and set $R_2 = \max(\gamma_{m'_2})_1 + 1$.
Clearly ($\star$) forces (i) to hold for $n = 2$.
Let $m_2$ such that $r_{m_2} > R_2 + 1$ and define $L_2 = \min(\gamma_{m_2})_1 - 1$. Using again ($\star\star$)
it follows that (i) holds for $n = 2$.
This procedure can be continued indefinitely and yields the sequences $\{L_n\}_n$ and $\{R_n\}_n$.
\end{proof}

The next two lemmas are obtained as corollaries of the previous proposition.
\begin{lemma}\label{lem:pasaporizquierda}
For every $L < 0$ there exists $R > 1$ such that if $\gamma$ is a hair
contained in $V^-$ and $(\gamma(1))_1 \ge R$ then $L \in (\gamma)_1$.
\end{lemma}
\begin{proof}
Fix $n$ so that $L_{n-1} \le L$ and define $R = R_n$.
Then $\gamma$ is a hair in $V^-$ whose endpoint is not on the left of $R_n$, that is, $(\gamma(1))_1 \ge R = R_n$.
By Proposition \ref{prop:}, $L_{n-1} \in (\gamma)_1$ and, consequently, $L \in (\gamma)_1$.
\end{proof}
\begin{lemma}\label{lem:pasaporderecha}
For every $R > 1$ there exists $L' < 0$ such that if $\gamma$ is a hair
contained in $V^+$ and $(\gamma(1))_1 \le L'$ then $R \in (\gamma)_1$.
\end{lemma}
In the setting of Lemma \ref{lem:pasaporderecha}, there is a cross-cut $c$ in $\tilde{U}_+$ which separates $\gamma(1)$ from
$\R \times \{1\}$ and such that $(c)_1 = \{R\}$.

\medskip

Let finish the proof of case (iii).
Apply Lemma \ref{lem:pasaporizquierda} to $L = -M$ (where $M$ comes from Lemma \ref{lem:aladerecha}) to obtain $R > 1$.
Consider the family
$$
\mathcal{A} = \{c \text{ is a cross-cut of } \tilde{U}_+ \text{ contained in some }(V + k)^-
\text{ such that } \min (c)_1 \ge R + k\}
$$
By definition, $\mathcal{A}$ is invariant by integer translations. It is not empty because $(V)_1$ is not bounded from above.
In addition, $\mathcal{A}$ is $\tilde{f}$-invariant.
Indeed, if $c \in \mathcal{A}$ then by Lemma \ref{lem:aladerecha} it is automatically contained in the region
$\{x \in \R \times [-1,1]: (\tilde{f}(x))_1 > (x)_1\}$. Consequently, $\min (\tilde{f}(c))_1 > \min (c)_1 \ge R + k$.
Since $c \subset (V + k)^-$, it then follows from Lemma \ref{lem:primeends} that $\tilde{f}(c)$ is contained
in $(V + k)^-$. Thus, $\tilde{f}(c) \in \mathcal{A}$.

Apply now Lemma \ref{lem:pasaporderecha} to $R$ to obtain $L'$.
Denote $x_n = \tilde{f}^n(x_0)$ the orbit of $x_0$.
Since $(x_n)_1 \to -\infty$ as $n$ tends to $-\infty$, there is $m > 0$ such that $(x_{-m})_1 \le L'$.
By Lemma \ref{lem:primeends},
the points $x_n$ are accessible from $V^+$ for $n \le 0$, so $x_{-m} + k$ is accessible from $V$ for some $k \le 0$.
Since $(x_{-m} + k)_1 \le L' + k \le L'$, remark after Lemma \ref{lem:pasaporderecha} provides
a cross-cut $c$ of $\tilde{U}_+$ in $V$ which separates $x_{-m} + k$ from $\R \times \{1\}$ and $\min (c)_1 = R$.
Thus, $c \in \mathcal{A}$ and $c_{-m} = c - k \in \mathcal{A}$.

The cross-cut $c_{-m} \in \mathcal{A}$ separates $x_{-m}$ from $\R \times \{1\}$ in $\tilde{U}_+$.
Thus, $\tilde{f}^{m}(c_{-m}) \in \mathcal{A}$ separates $\tilde{f}^{m}(x_{-m}) = x_0$ from $\R \times \{1\}$.
However, the arc $\beta$ joins $x_0$ and $\R \times \{1\}$ and does not meet any element of $\mathcal{A}$
because $(\beta)_1 = \{0\}$ is disjoint to $[R, +\infty)$. \qed

\end{document}